\documentclass[12 pt]{extarticle}
\usepackage[utf8]{inputenc}
\usepackage[letterpaper, total={7in, 9in}]{geometry}

\usepackage[normalem]{ulem}

\makeatletter
\newcommand{\customlabel}[2]{#2\def\@currentlabel{#2}\label{#1}}

\makeatother

\usepackage[backend=biber, style=trad-abbrv, sorting=nyt, maxnames=100,backref=true, giveninits=true, sortcites=true, url=false, doi=false]{biblatex}
\usepackage{amsmath}
\usepackage{amsthm}
\usepackage{amssymb}
\usepackage{mathtools}
\usepackage[dvipsnames]{xcolor}
\usepackage[hidelinks]{hyperref}
\usepackage{enumitem}
\usepackage{cleveref}
\usepackage{tikz}
\usetikzlibrary {graphs} 
\usepackage{caption}

\def\eps{\varepsilon}

\newtheorem{thm}{Theorem}[section]

\newtheorem{quest}[thm]{Question}
\newtheorem{prop}[thm]{Proposition}
\newtheorem{claim}{Claim}
\newtheorem{lemma}[thm]{Lemma}
\newtheorem{cor}[thm]{Corollary}

\hypersetup{colorlinks=true,citecolor=blue,linktoc=page}
\hypersetup{%
  colorlinks,
  linkcolor={red!50!black},
  citecolor={green!50!black},
  urlcolor={blue!50!black}
}

\title{Even smaller universal posets}

\author{
{J\'ozsef Balogh\thanks{Department of Mathematics, University of Illinois Urbana-Champaign, Urbana, IL, USA, and Extremal Combinatorics and Probability Group (ECOPRO), Institute for Basic Science (IBS), Daejeon, South Korea. Email: \texttt{jobal@illinois.edu}. Partially supported by NSF grants RTG DMS-1937241, FRG DMS-2152488, UIUC  Campus Research Board Award RB26026, the  Simons Collaboration grant [SFI-MPS-TSM-00013107, JB], and the Institute for Basic Science (IBS-R029-C4).}
}, Ramon I. Garcia\thanks{Department of Mathematics, University of Illinois Urbana-champaign, Urbana, IL, USA. Email: \texttt{rig2@illinois.edu}. Partially suported by NSF grant RTG DMS-1937241, FRG DMS-2152488 and the R.H. Schark Fellowship.}, 
Marcelo Sales \thanks{Department of Mathematics, University of California, Irivne, CA, USA. Email:\texttt{mtsales@uci.edu}. Supported by US Air force grant FA9550-23-1-0298.
}}

\date{}

\begin{document}

\maketitle

\begin{abstract}
We show that for every $\eta>0$ and sufficiently large $n$, there exists a poset of size $2^{(1+\eta)n/2}$ containing all the $n$-element posets as induced subposets. This improves a recent result of Bastide, Groenland and Nenadov~\cite{BastidePosets2025}.

 Our proof provides  a labeling scheme preserving transitivity, inspired by the Boolean lattice. Among other tools, we use the Szemer\'edi Regularity Lemma.
\end{abstract}

\section{Introduction}

A partially ordered set, or \emph{poset}, is a set equipped with a reflexive, antisymmetric, and transitive relation. Throughout the paper, containment of posets is always meant in the induced sense. More precisely, if $P$ and $Q$ are posets, an injective map $\psi:P\to Q$ is an \emph{order embedding} if, for all $x,y\in P$,
\begin{align*}
        x\le_P y \quad\quad\quad \text{if and only if}\quad\quad\quad \psi(x)\le_Q \psi(y).
\end{align*}
Given a family $\mathcal P$ of posets, we say that a poset $Q$ is \emph{universal} for $\mathcal P$ if every $P\in\mathcal P$ admits an order embedding into $Q$.

Universal objects are a classical topic in combinatorics. For graphs, the question was first studied in the infinite setting. Rado~\cite{Rado1964} constructed a countably infinite graph, now known as the \emph{Rado graph}, which contains every countable graph as an induced subgraph. In the finite setting, Moon~\cite{Moon1965} asked the quantitative question of determining the smallest size $g(n)$ of a universal graph containing every graph on $n$ vertices as an induced subgraph. In~\cite{Moon1965}, it was shown that $g(n)=O(n2^{n/2})$, while a simple counting argument gives $g(n)\geq 2^{(n-1)/2}$. This was later improved by several authors and finally settled by Alon~\cite{Noga2017asymptotically}, who showed that $g(n)=(1+o(1))2^{(n-1)/2}$. More generally, there is a vast literature on the asymptotics of the minimum possible number of vertices in a universal graph for a given family of $n$-vertex graphs or digraphs; see, for instance,
\cites{BollobasThomason1981,Chung1990,KannanNaorRudich1992,AlstrupKaplanThorupZwick2015,AlonNenadov2019,DujmovicEsperetGavoilleJoretMicekMorin2021,bonamy2021optimal}.

For posets, the analogous question goes back at least as far. The existence of a countably infinite universal poset containing every countable poset as an induced subposet was proved in several early works; see \cites{Johnston1956,Jonsson1956,Fraisse1986}. In this paper, we focus on quantitative bounds in the finite setting. In particular, we are interested in the following question, raised explicitly by Hamkins~\cite{Hamkins2010}. Let $\mathcal P_n$ denote the family of all posets with $n$ elements, and let
\begin{align*}
        f(n)=\min\{|Q|: Q \text{ is universal for }\mathcal P_n\}
\end{align*}
be the minimum size of a universal poset for this family.

\begin{quest}\label{question}
    What is the asymptotic value of $f(n)$?
\end{quest}

The best known lower bound for Question~\ref{question} comes from a simple counting argument. A classical result of Kleitman and Rothschild~\cite{Kleitman1975Posets} shows that there are $|\mathcal{P}_n|=2^{(1+o(1))n^2/4}$ distinct $n$-element posets. Hence, a universal poset $Q$ must satisfy $\binom{|Q|}{n}\geq 2^{(1+o(1))n^2/4}$, which gives the lower bound $f(n)\geq 2^{(1+o(1))n/4}$. The most elementary upper bound is $f(n)\leq 2^n$, obtained by noting that the Boolean lattice $(2^{[n]},\subseteq )$ is universal for $\mathcal{P}_n$. Indeed, every $n$-element poset $P$ embeds into the Boolean lattice $(2^P,\subseteq)$ by sending an element $x$ to its closed down-set $D_x=\{y\in P:y\le_P x\}$.

There is some evidence that the lower bound $2^{(1+o(1))n/4}$ should be the correct order of magnitude. Given a poset $P$, its \emph{comparability graph} $G_P$ is the graph with vertex set $P$, in which two distinct elements are adjacent if and only if they are comparable in $P$. If $Q$ is universal for all $n$-element posets, then $G_Q$ contains, as an induced subgraph, the comparability graph of every $n$-element poset. Thus the problem of finding small universal posets is closely related to the problem of finding small induced-universal graphs for the class of comparability graphs. Bonamy, Esperet, Groenland, and Scott~\cite{bonamy2021optimal} proved that the class of $n$-vertex comparability graphs admits an induced-universal graph on $2^{(1+o(1))n/4}$ vertices, and that this is optimal up to the lower-order term. Although this result suggests that the exponent $n/4$ is the natural target for universal posets, as already mentioned in their work, the additional difficulty in the poset setting is that a universal graph for comparability graphs need not itself be the comparability graph of a universal poset: one must encode not only which pairs are comparable, but also the orientation of each comparison in a way that is globally transitive.

The first exponential improvement over the Boolean-lattice upper bound was due to Bastide, Groenland, and Nenadov~\cite{BastidePosets2025}, who showed that there exists a subposet of the Boolean lattice of size at most $2^{(1+o(1))2n/3}$ which is universal for $\mathcal{P}_n$. Our main result improves this bound further.

\begin{thm}\label{thm:universal-bound}
For every $\eta \in (0,1)$, there exists $n_0:=n_0(\eta)$ such that for $n\geq n_0$ there exists a poset $Q_n$ which is universal for $\mathcal P_n$ of size
\begin{align*}
    |Q_n|\leq 2^{n/2+\eta n}.
\end{align*}
\end{thm}

In particular, Theorem~\ref{thm:universal-bound} implies that $f(n)\leq 2^{(1+o(1))n/2}$ for sufficiently large $n$. Our proof takes a different route from that of~\cite{BastidePosets2025}. Instead of finding a suitable subposet of the Boolean lattice $(2^{[n]}, \subseteq)$, we analyze the structure of the comparability graph $G_P$ of a poset $P$ and convert an adjacency labeling scheme for $G_P$ into a labeling scheme that preserves transitivity in $P$. Our approach is similar to that of~\cite{bonamy2021optimal} in that we use the regularity lemma to classify the pairs in the reduced graph of $G_P$ according to their densities. The main new contribution is a labeling scheme preserving transitivity, inspired by the Boolean lattice; see Section~\ref{sec: poset construction}.

\medskip
\noindent\textbf{Organization and notation.}
The paper is organized as follows. Section~\ref{sec: poset construction} introduces the labeling scheme used in the proof of Theorem~\ref{thm:universal-bound}. In Section~\ref{reglemma}, we state the Szemer\'edi Regularity Lemma and prove the structural results needed for the proof, while Section~\ref{sec:proof} contains the proof of Theorem~\ref{thm:universal-bound}.

For a poset $Q$ and $A\subseteq Q$, we write $Q[A]$ for the subposet induced by $A$, i.e., the poset with ground set $A$ and order relation inherited from $Q$. A poset $Q$ is $k$-\emph{layered} if its ground set admits a partition $Q=L_1\cup\cdots\cup L_k$ such that whenever $u\in L_i$, $v\in L_j$, and $i<j$, we do not have $v\le_Q u$.

\section{A poset construction}\label{sec: poset construction}

In this section we introduce the poset construction that will be used later in the proof of Theorem~\ref{thm:universal-bound}. The starting point for our construction is the standard embedding of a poset into a Boolean lattice. Given a poset $P$, one may send each element $x\in P$ to its closed downset $\overline{D}_x^P=\{y\in P:y\le_P x\}$. Note that $x\le_P y$ if and only if $\overline{D}_x^P\subseteq  \overline{D}_y^P$. Therefore, every $n$-element poset embeds into the Boolean lattice $(2^{[n]},\subseteq)$. Equivalently, one can think of this as a labeling scheme in which each vertex stores all of its relations to vertices that are smaller or equal than itself. The drawback of this construction is that is too expensive for our purposes: For any set $S\subseteq [n]$, there exists an element $x \in S$ and a poset $P\in \mathcal{P}_n$ such that $\overline{D}_x^P=S$.

Our solution is to compress the above version of the Boolean-lattice labeling. For each poset $P$, we choose an auxiliary set $S:=\varphi(P)\subseteq P$. The set $S$ should be thought of as a collection of vertices through which many comparisons will be certified. Instead of storing the full closed upset and downset of every element, we store full information only for vertices outside $S$, while vertices in $S$ store only their information on $S$. The order in the universal poset is then defined by comparing these truncated upsets and downsets. Thus, if the possible choices for $S$ are not too many, and if for each fixed $S$ there are few possible truncated upsets and downsets, the resulting universal poset is much smaller than the original Boolean lattice.

To be more precise, let $\mathcal P$ be a family of $n$-element posets. For each $P\in\mathcal P$, we choose a bijection $\psi_P$ between its ground set and $[n]$, and identify $P$ with the corresponding labelled poset on $[n]$. We also assume that, for every $P\in\mathcal P$, we are given a subset $\varphi(P)\subseteq [n]$. The choice of $\varphi(P)$ may depend on $P$. For $x\in [n]$, write
\begin{align*}
        \overline{U}_x^P=\{y\in [n]:\psi_P^{-1}(x)\le_P \psi_P^{-1}(y)\}
        \quad\quad\quad \text{and} \quad\quad\quad
        \overline{D}_x^P=\{y\in [n]:\psi_P^{-1}(y)\le_P \psi_P^{-1}(x)\}
\end{align*}
for the closed upset and closed downset of $x$ in $[n]$. We then define the \emph{compressed upset and downset} of $x$ by
\begin{align}\label{defi: UxDx}
        U_x^P=
        \begin{cases}
        \overline U_x^P, & x\notin \varphi(P),\\
        \overline U_x^P\cap \varphi(P), & x\in \varphi(P),
        \end{cases}
        \quad\quad\quad \text{and} \quad\quad\quad
        D_x^P=
        \begin{cases}
        \overline D_x^P, & x\notin \varphi(P),\\
        \overline D_x^P\cap \varphi(P), & x\in \varphi(P).
        \end{cases}
\end{align}
Thus a vertex outside $\varphi(P)$ stores its full closed upset and downset, whereas a vertex inside $\varphi(P)$ stores only their intersection on $\varphi(P)$. We define the label of $x$ in $P$ to be
\begin{align}\label{eq:labeling}
        L_P(x)=\big(x,\varphi(P),U_x^P,D_x^P\big).
\end{align}
Let $Q=Q(\mathcal P,\varphi)$ be the set of all labels that arise in this way. That is,
\begin{align*}
        Q=\left\{
        L_P(x):
        P\in\mathcal P,\ x\in [n]
        \right\}.
\end{align*}
We define an order relation $\leq_Q$ on the elements of $Q$ as follows. Let
\begin{align*}
        \ell=(i,S,U,D),
        \quad \quad\quad\text{and} \quad \quad\quad
        \ell'=(j,S',U',D')
\end{align*}
be two elements of $Q$. We declare $\ell\le_Q \ell'$ if either \(\ell=\ell'\), or all the following conditions hold:
\begin{enumerate}[label=(Q\arabic*)]
    \item\label{eq:same_S} $S=S'$.
    \item\label{eq:trace-nesting} Their traces on the set $S$ and $S'$ satisfy $U' \cap S'\subseteq U\cap S$ and $D\cap S\subseteq D'\cap S'$.
    \item\label{eq:specialcase} If $i \notin S$ and $j\notin S'$, then  one of the next two conditions hold:
    \begin{enumerate}[label=(Q3.\roman*)]
        \item\label{eq:full-nesting} The full closed downsets and upsets satisfy $U'\subseteq U$ and $D\subseteq D'$.
        \item\label{eq:S-mediated} It holds that $(U\cap S)\cap (D'\cap S')\neq \emptyset$.
    \end{enumerate}
\end{enumerate}

We will show that this relation $\leq_Q$ defines a poset $Q$ and that $Q$ is universal for $\mathcal P$. We first state a simple observation that will be used repeatedly. By construction, we have
\begin{align}\label{eq:unique_intersection}
U\cap D=\{i\}
\end{align}
for every label $\ell=(i,S,U,D)$.

\begin{prop}\label{prop: Universal Q}
The relation $\leq_Q$ defined above is a partial order on $Q$.
\end{prop}

\begin{proof}
Reflexivity follows directly from the definition of $\leq_Q$. We proceed to verify that anti-symmetry holds. Let $\ell=(i,S,U,D)$ and $\ell'=(j,S',U',D')$ be two labels in $Q$. Assume that $\ell'\leq_Q \ell$ and $\ell\leq_Q\ell'$. From condition \ref{eq:same_S} we obtain that $S=S'$. 

We claim that $i=j$, $U=U'$ and $D=D'$. Indeed, by condition \ref{eq:trace-nesting}, the fact that $\ell \leq_Q\ell'$, $\ell'\leq_Q \ell$ and $S=S'$ we obtain that
\begin{align}\label{eq:equal_trace}
    U\cap S=U'\cap S \quad\quad\quad \text{and} \quad\quad\quad D\cap S=D'\cap S.
\end{align}
Thus, in particular $U\cap D\cap S=U'\cap D'\cap S$ holds. We split the proof into two cases:

\vspace{0.2cm}

\noindent \textbf{Case 1:} Either $i \in S$ or $j\in S$.

\vspace{0.1cm}

We may suppose that $i\in S$. Therefore, by~\eqref{eq:unique_intersection} we have $i \in U\cap D\cap S=U'\cap D'\cap S\subseteq U'\cap D'$. However, this implies again by~\eqref{eq:unique_intersection} that $i=j$. Since $i=j \in S$, we have by~\eqref{eq:equal_trace} that
\begin{align*}
    U=U\cap S=U'\cap S=U' \quad\quad\quad \text{and} \quad\quad\quad D=D\cap S=D'\cap S=D'.
\end{align*}
This concludes that $\ell=\ell'$.

\vspace{0.2cm}

\noindent \textbf{Case 2:}  $i, j\notin S$.

\vspace{0.1cm}

In this case, we have to use the extra condition \ref{eq:specialcase}. Note that by the hypothesis it always hold that $(U\cap S)\cap (D'\cap S)= (U'\cap S)\cap (D\cap S)=\emptyset$. Therefore, condition \ref{eq:S-mediated} can never happen and we obtain that $\ell\leq_Q\ell'$ and $\ell'\leq_Q\ell$ via \ref{eq:full-nesting}. The condition in both directions gives us that $U=U'$ and $D=D'$. We finish by noting that~\eqref{eq:unique_intersection} implies that $x=y$. Hence, $\ell=\ell'$. 

\vspace{0.2cm}

It remains to prove transitivity. Let $\ell_k=(i_k, S_k, U_k,D_k)$ for $1\leq k\leq 3$. Suppose that $\ell_1\leq_Q \ell_2$ and $\ell_2\leq_Q\ell_3$. Our goal is to show that $\ell_1\leq_Q\ell_3$. First, note that condition \ref{eq:same_S} implies that $S=S_1=S_2=S_3$. Moreover, condition \ref{eq:trace-nesting} together with $\ell_1\leq_Q \ell_2$ and $\ell_2\leq_Q \ell_3$ implies that
\begin{align*}
        U_3\cap S\subseteq U_2\cap S\subseteq U_1\cap S
        \quad\quad\quad \text{and} \quad\quad\quad D_1\cap S\subseteq D_2\cap S\subseteq D_3\cap S.
\end{align*}
Hence, \ref{eq:trace-nesting} holds for $\ell_1$ and $\ell_3$. 

If at least one of $i_1$, $i_3$ lies on $S$, then the last two conditions already show that $\ell_1\leq_Q\ell_3$. Hence, we may assume that $i_1,i_3\notin S$ and we need to verify the additional condition \ref{eq:specialcase} for two vertices outside $S$. We again split the proof into three cases:

\vspace{0.2cm}
\noindent \textbf{Case A:} $i_2\in S$.

\vspace{0.2cm}

Now $i_2\in S$ implies that  $i_2\in U_2\cap S$ and $i_2\in D_2\cap S$. Since $\ell_1\le_Q\ell_2$, we have by~\ref{eq:trace-nesting} that $U_2\cap S\subseteq U_1\cap S$ and hence $i_2\in U_1\cap S$. Similarly, since $\ell_2\le_Q\ell_3$, we also have by~\ref{eq:trace-nesting} that $D_2\cap S\subseteq D_3\cap S$ and consequently that $i_2\in D_3\cap S$. Thus, 
\begin{align*}
    i_2\in (U_1\cap S)\cap(D_3\cap S)
\end{align*}
and by~\ref{eq:S-mediated} we obtain that $\ell_1\leq_Q\ell_3$.

\vspace{0.2cm}

\noindent \textbf{Case B:} $i_2\notin S$ and  both comparisons $\ell_1\leq_Q \ell _2$ and $\ell_2\leq_Q \ell_3$ use condition \ref{eq:full-nesting}.

\vspace{0.2cm}

In this case, it holds that
\begin{align*}
     U_3\subseteq U_2\subseteq U_1
       \quad\quad \quad\text{and}\quad\quad\quad
        D_1\subseteq D_2\subseteq D_3.
\end{align*}
Hence,~\ref{eq:full-nesting} holds for $\ell_1$ and $\ell_3$ and we have that $\ell_1\leq_Q\ell_3$.

\vspace{0.2cm}

\noindent \textbf{Case C:} $i_2\notin S$ and  at least one of the comparisons $\ell_1\leq_Q\ell_2$ and $\ell_2\leq_Q\ell_3$ uses condition \ref{eq:S-mediated}.

\vspace{0.2cm}

Suppose that $\ell_1\le_Q\ell_2$ uses condition \ref{eq:S-mediated}. Then there is some $s\in (U_1\cap S)\cap(D_2\cap S)$ mediating the comparison. Since, by $\ell_2\leq_Q\ell_3$, condition~\ref{eq:trace-nesting} implies  that $D_2\cap S\subseteq D_3\cap S$, we obtain that $s\in (U_1\cap S)\cap(D_3\cap S)$.
This implies that~\ref{eq:S-mediated} holds and $\ell_1\leq_Q \ell_3$. A similar argument holds if $\ell_2\le_Q\ell_3$ uses condition \ref{eq:S-mediated}.

\vspace{0.2cm}

This proves transitivity and concludes the proof of the lemma.
\end{proof}

We now prove that the construction contains every poset in $\mathcal P$ as an induced subposet.

\begin{prop}\label{prop:Embedding}
For every $P\in\mathcal P$, and every  $\varphi(P)$, the map $\psi_P^{-1}(x)\mapsto L_P(x)$
is an order embedding of $P$ into $Q$.
\end{prop}

\begin{proof}
Fix $P\in\mathcal P$, and write $S=\varphi(P)$. Let $x,y\in [n]$. We want to show that $\psi_P^{-1}(x)\leq_P\psi_P^{-1}(y)$ if and only if $L_P(x)\leq_Q L_P(y)$. First suppose that $\psi_P^{-1}(x)\leq_P\psi_P^{-1}(y)$. Then by definition, the closed downsets and upsets satisfies
\begin{align*}
        \overline U_y^P\subseteq \overline U_x^P
        \quad\quad\quad\text{and}\quad\quad\quad
        \overline D_x^P\subseteq \overline D_y^P.
\end{align*}
Hence, it follows immediately that
\begin{align*}
        U_y^P\cap S\subseteq U_x^P\cap S
        \quad\quad\quad\text{and}\quad\quad\quad
        D_x^P\cap S\subseteq D_y^P\cap S.
\end{align*}
If at least one of $x,y$ lies in $S$, then by the definition of $\leq_Q$ it holds that $L_P(x)\le_Q L_P(y)$. If both $x,y\notin S$, then $U_x^P=\overline U_x^P$, $U_y^P=\overline U_y^P$, $D_x^P=\overline D_x^P$, and $D_y^P=\overline D_y^P$. Hence condition \ref{eq:full-nesting} holds, and again $L_P(x)\le_Q L_P(y)$.

Conversely, suppose that $L_P(x)\le_Q L_P(y)$. We claim that $\psi_P^{-1}(x)\leq_P\psi_P^{-1}(y)$. We split the proof into two cases:

\vspace{0.2cm}

\noindent \textbf{Case 1:} Either $x\in S$ or $y\in S$.

\vspace{0.2cm}

Suppose that $x\in S$. Then by condition \ref{eq:trace-nesting} it holds that $x\in D_x^P\cap S\subseteq D_y^P\cap S$. Therefore, $x\in D_y^P$ and by definition we have that $\psi_P^{-1}(x)\leq_P\psi_P^{-1}(y)$. Similarly if $y\in S$, then $y\in U_y^P\cap S\subseteq U_x^P\cap S$. Thus $y\in U_x^P$, and again $\psi_P^{-1}(x)\leq_P\psi_P^{-1}(y)$.

\vspace{0.2cm}

\noindent \textbf{Case 2:} $x,y\notin S$.

\vspace{0.2cm}

If the comparison uses condition \ref{eq:full-nesting}, then
$y\in U_y^P\subseteq U_x^P$, and hence $\psi_P^{-1}(x)\leq_P\psi_P^{-1}(y)$. If instead the comparison used condition \ref{eq:S-mediated}, then there exists some $s \in S$ such that $s\in (U_x^P\cap S)\cap(D_y^P\cap S)$. Thus, by definition, $\psi_P^{-1}(x)\leq_P\psi_P^{-1}(s)$ and $\psi_P^{-1}(s)\leq_P\psi_P^{-1}(y)$. Hence, by transitivity of $P$, we obtain $\psi_P^{-1}(x)\leq_P\psi_P^{-1}(y)$. This concludes the proof of the proposition.
\end{proof}

We finish the section by giving an estimate of the size of the universal poset obtained from this construction. Let
\begin{align*}
        \mathcal{S}=\{\varphi(P):P\in\mathcal P\}
\end{align*}
be the collection of auxiliary sets $\varphi(P)$ that arise from the family of posets in $\mathcal P$. For each $S\in\mathcal{S}$ and $x\in [n]$, define the set of possible local profiles by
\begin{align*}
        \mathcal T_{S,x}=
        \left\{
        (U_x^P,D_x^P):
        P\in\mathcal P,\ \varphi(P)=S,\ x\in [n]
        \right\}.
\end{align*}
Thus $\mathcal T_{S,x}$  stores all possible compressed upsets and downsets once the auxiliary set $S$ and an element $x\in [n]$ has been fixed.

\begin{prop}\label{prop:construction_size}
The poset $Q:=Q(\mathcal P,\varphi)$ constructed above satisfies
\begin{align*}
        |Q|
        =
        \sum_{S\in\mathcal{S}}\sum_{x\in [n]}|\mathcal T_{S,x}|
        \leq        n\cdot |\mathcal{S}|\cdot\max_{S\in\mathcal{S},\, x\in[n]}|\mathcal T_{S,x}|.
\end{align*}
\end{prop}

\begin{proof}
By definition, the elements of $Q$ are precisely the set of labels
\begin{align*}
        L_P(x)=(x,S,U_x^P,D_x^P)
\end{align*}
arising from a poset $P\in\mathcal P$. Grouping these labels according to the auxiliary set $S=\varphi(P)$ and $x\in [n]$, we obtain the desired result.
\end{proof}

As a consequence of Propositions~\ref{prop: Universal Q}, \ref{prop:Embedding}, and~\ref{prop:construction_size}, the task of constructing a small universal poset for a family $\mathcal P$ is reduced to choosing the auxiliary sets $\varphi(P)$ in such a way that both the number of possible choices for $\varphi(P)$ and the number of possible local profiles $(U_x^P,D_x^P)$ are small. This will be done in Section~\ref{sec:proof}.

\section{Regularity Lemma}\label{reglemma}

In this section, we state a refined version of the Szemer\'edi Regularity lemma that will be used in the proof of Theorem \ref{thm:universal-bound}. This version will be
convenient later, since we will start with a partition coming from the layers of a poset
and then refine it into a regular partition. We then recall an induced embedding lemma
of Alon, Balogh, Bollobás, and Morris~\cite{Alon2011Hereditary} and use it to obtain an important structural consequence: the reduced graph formed by the
regular pairs of medium density in the comparability graph of a poset is triangle-free. 

Let $G$ be a graph, and let $A,B\subseteq V(G)$ be disjoint non-empty vertex sets. The \emph{edge density} of the pair $A,B$ in $G$ is defined by $d_G(A,B)=\frac{e_G(A,B)}{|A||B|}$, where $e_G(A,B)$ denotes the number of edges in $G$ with one endpoint in $A$ and the other in $B$. When the graph $G$ is clear from the context, we
write simply $e(A,B)$ and $d(A,B)$. Given $\eps>0$, we say that $(A,B)$ is  $\varepsilon$-\emph{regular} if for every pair of subsets $X\subseteq A$ and $Y\subseteq B$ satisfying $|X|\geq \varepsilon |A|$ and $|Y|\geq \varepsilon |B|$ we have $|d(X, Y)-d(A,B)|< \varepsilon$. 

We next define the form of regular partition that we will use. Let $G$ be a graph on $n$ vertices, and let $\varepsilon>0$. An \emph{$\eps$-regular partition} of a graph $G$ is a partition $V(G)=V_0\cup \ldots \cup V_t$ such that
\begin{enumerate}[label=($\roman*$),itemsep=-2pt]
\item\label{exceptionalvtx} $|V_0|\leq \varepsilon n$.
\item\label{equitablepart} $|V_1|=\ldots =|V_t|.$
\item\label{allbutt2} All but at most $\varepsilon t^2$ pairs of distinct indices $i,j\in[t]$ satisfy that $(V_i,V_j)$ is $\varepsilon$-regular.
\end{enumerate}
The vertex set $V_0$ is usually called the {\it exceptional set}.

A partition $\{V_1,\ldots,V_t\}$ is a \emph{refinement} of the partition $\{W_1,\ldots,W_k\}$ if, for every $i\in[t]$ there exists $j\in[k]$ such that $V_i\subseteq W_j$. We use the following version of the Szemer\'edi Regularity Lemma~\cite{Szemeredi1978}, in which a prescribed initial partition is refined by the resulting regular partition. This formulation is obtained by initiating the standard proof with the given partition (e.g,~\cite[Theorem~5.1]{Schaudt2019Partitioning}).

\begin{thm}[Regularity lemma]\label{thm:refined_regularity}
For every $\varepsilon>0$ and $m,k\in\mathbb{N}$, there exists integer $M:=M(\eps,m,k)$ such that the following holds for every graph $G$ on $n\geq M$ vertices, and with a partition $V(G)=W_1\cup\ldots\cup W_k$ of its vertex set: There exists an $\varepsilon$-regular partition $V(G)=V_0\cup\ldots\cup V_t$ with $\max\{m,k\}\leq t\leq M$ and such that for every $i\in [t]$, there exists $\omega(i)\in [k]$ with $V_i\subseteq W_{\omega(i)}$. That is, the partition $V(G)=\bigcup_{x\in V_0}\{x\}\cup \bigcup_{i=1}^t V_i$ is a refinement of $\{W_1,\ldots, W_k\}$.
\end{thm}

We now recall an induced embedding lemma of Alon, Balogh, Bollobás, and
Morris~\cite{Alon2011Hereditary} used in the context of hereditary properties. For $a, b\in\mathbb{N}$, let $\mathcal{H}(a,b)$ denote the class of graphs whose vertex set can be partitioned into $a$ complete graphs and $b$ independent sets. The following version is equivalent to the original statement by relabeling the graphs.

\begin{lemma}[\cite{Alon2011Hereditary}, Lemma 10]\label{lem: regularity and hereditary}
    Given $\delta>0$ and $m,r\in \mathbb{N}$, there exists an $\varepsilon_0:=\eps_0(\delta, m,r)$ and $n_0:=n_0(\delta,m,r)\in\mathbb{N}$ such that the following holds for $n\geq n_0$ and $\eps\leq \eps_0$. Let $G$ be a graph with a vertex partition $V(G)=V_1\cup\ldots\cup V_{r}$ such that $|V_i|=n$ for all $i\in [r]$, and the pair $(V_i,V_j)$ is $\varepsilon$-regular and with density satisfying $\delta\leq d(V_i,V_j)\leq 1-\delta$ for all distinct $i,j\in[r]$. Then, there exists $a,b\in\mathbb{Z}_{\geq 0}$ with $a+b=r$ such that $G$ contains, as induced subgraphs, all elements from $\mathcal{H}(a,b)$ on at most $m$ vertices.
\end{lemma}

We finish this section with the structural statement that will be used in the proof of Theorem~\ref{thm:universal-bound}. Let $F$ be the graph with vertex set $\{v_1, v_2, v_3, v_4, v_5, v_6\}$ and edge set 
\[
E(F)=\{v_1 v_2, v_2v_3, v_3 v_1, v_4 v_1, v_2 v_5, v_3 v_6\}.\] 
That is, the graph $F$ consists of the triangle $v_1 v_2 v_3$ with $3$ additional vertices $v_4$, $v_5$, $v_6$, each one is attached to a different vertex of the triangle (see Figure~\ref{fig:Graph F}). The graph $F$ will play a special role in our argument because it is a classical forbidden induced subgraph for comparability graphs: if $G_P$ is the comparability graph of a poset $P$, then $G_P$ contains no induced copy of $F$ (see~\cite{Gilmore1964Comparability}).

\begin{figure}[h]

\begin{center}
\begin{tikzpicture}[scale=0.75]
\def\x{3}

\draw (.8,0)node[circle,draw, inner sep=\x, fill,label=right:$v_2$](v2){}
 (-.8,0)node[circle,,draw, inner sep=\x, fill,label=left:$v_3$](v3){}
(1.6,-.8)node[circle,draw, inner sep=\x, fill,label=right:$v_5$](v5){}
 (-1.6,-.8)node[circle,draw, inner sep=\x, fill,label=left:$v_6$](v6){}
 (0,1.2)node[circle,draw, inner sep=\x, fill,label=left:$v_1$](v1){}
 (0,2.4)node[circle,draw, inner sep=\x, fill,label=$v_4$](v4){};

 \draw[black,line width=1](v1)--(v2) (v2)--(v3) (v3)--(v1) (v1)--(v4) (v2)--(v5) (v3)--(v6);
\end{tikzpicture}
    \caption{The graph $F$.}\label{fig:Graph F}
\end{center}
\end{figure}
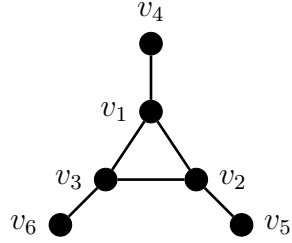

As a corollary of Lemma \ref{lem: regularity and hereditary} we obtain the following result.

\begin{cor}\label{cor:triangle-free}
    For every $\delta>0$, there exists $\varepsilon_0$ and $n_0:=n_0(\delta)\in\mathbb{N}$ such that the following holds for $n\geq n_0$ and $\eps\leq \eps_0$. Let $G$ be a graph with vertex partition $V(G)=V_1\cup V_2 \cup V_{3}$ such that $|V_i|\geq n_0$ for every $i\in [3]$, and the pair $(V_i,V_j)$ is $\varepsilon$-regular with density $\delta\leq d(V_i,V_j)\leq 1-\delta$ for every pair of distinct $i,j\in[3]$. Then, $G$ contains $F$ as an induced subgraph.
\end{cor}

\begin{proof}
    By Lemma~\ref{lem: regularity and hereditary}, it suffices to show that $F\in\mathcal{H}(a,b)$ for every pair $a,b\in\mathbb{Z}_{\geq 0}$ satisfying $a+b=3$. 
    Indeed, we can partition $F$ into the following four possibilities:
    \begin{itemize}
        \item[\tiny$\bullet$] Three disjoint edges $\{v_1,v_4\}, \{v_2,v_5\}, \{v_3,v_6\}$;
        \item[\tiny$\bullet$] Two disjoint edges $\{v_2,v_3\}$, $\{v_1,v_4\}$ and an independent set $\{v_5, v_6\}$;
        \item[\tiny$\bullet$] An edge $\{v_1,v_4\}$ and two independent sets $\{v_2,v_6\},\{v_3,v_5\}$;
        \item[\tiny$\bullet$] Three independent sets $\{v_1,v_5\},\{v_2,v_6\},\{v_3,v_4\}$.
    \end{itemize}
    Each possibility witness the containment for a different pair $\{a,b\}$ satisfying $a+b=3$.
\end{proof}

An immediate consequence of Corollary \ref{cor:triangle-free} is that the reduced graph obtained by considering the pairs of density between $\delta$ and $1-\delta$ is triangle-free. This is going to be crucial for our construction in the next section.

\section{Proof of Theorem~\ref{thm:universal-bound}}\label{sec:proof}

In this section, we prove Theorem~\ref{thm:universal-bound}. The proof uses the abstract
construction from Section~\ref{sec: poset construction}. For each poset $P$, our
task is to choose a labeling of its ground set by $[n]$ and an auxiliary set
$\varphi(P)\subseteq[n]$ in such a way that both the number of possible auxiliary sets  and the number of possible compressed profiles $(U_x^P,D_x^P)$ are small. To obtain such a good compression, similarly as in \cite{bonamy2021optimal}, we will use the regularity lemma to analyze the structure of the comparability graph $G_P$ of $P$.

Before we proceed with the details, we fixed the parameters throughout the proof. Let $\eta>0$ be the real number given in the statement. We choose real numbers $\eps, \delta \in (0,1)$ and $k\in \mathbb{N}$ satisfying
\begin{align*}
        n^{-1}\ll \eps\ll \delta \ll k^{-1}\ll \eta .
\end{align*}
More precisely, the parameter $\eps$ is chosen small enough so that Corollary~\ref{cor:triangle-free} applies with density parameter $\delta$.

Let $P$ be a poset of size $n$. We first associate to $P$ a layering of antichains. For each $a\in P$, let $r(a)$ be the maximum length of a chain in $P$ whose largest
element is $a$. Let $h$ be the size of the maximum chain in $P$. For each $1\leq r\leq h$, define 
\begin{align*}
    L_r=\{a\in P:r(a)=r\}
\end{align*}
to be the set of elements whose longest chain ending at them has length exactly $r$. Then for $1\leq r\leq h$, the layer $L_r$ is an antichain and it partitions $P$ into $P=L_1\cup L_2\cup\cdots\cup L_h$. Moreover, if $a\in L_i$,
$b\in L_j$, and $i<j$, then we cannot have $b<_P a$.

We now form an auxiliary partition $P=W_1\cup\ldots\cup W_k$ as follows. Consider a total ordering $\prec$ of the elements of $P$ such that $L_1\prec\ldots\prec L_h$, i.e., if $a\in L_i$, $b\in L_j$ and $i<j$, then $a\prec b$. For each $1\leq i \leq h$, we choose the ordering of the elements inside $L_i$ arbitrarily. We split the elements of $P$ into $k$ consecutive blocks $W_1,\ldots,W_k$ of sizes as equal as possible, respecting the ordering $\prec$. Thus, it holds that
\begin{align}\label{eq:sizeW_i}
        \left\lfloor\frac{n}{k}\right\rfloor\leq |W_i|\leq \left\lceil \frac nk\right\rceil
\end{align}
for every $1\leq i \leq k$. Moreover, by respecting the ordering $\prec$, we have the following important property: 
\begin{align}\label{eq:star_property}
\text{If $i<j$, $a\in W_i$, $b\in W_j$, and $a$ and $b$ are
comparable in $P$, then $a<_P b$.}
\end{align}

We apply Theorem~\ref{thm:refined_regularity} to the comparability graph $G_P$, with
regularity parameter $\eps$ and initial partition
$W_1\cup\cdots\cup W_k$. We then obtain an $\eps$-regular partition $V(G_P)=V_0\cup V_1\cup\cdots\cup V_t$,
where $k\leq t\leq M$, for some constant $M:=M(\eps,k,k)$. The sets $V_1,\ldots,V_t$ have the same size, and for every $i\in[t]$ there exists a unique
$\omega(i)\in[k]$ such that $V_i\subseteq W_{\omega(i)}$.

We now define a coloring of the pairs $[t]^{(2)}$. For $i\neq j$, color the pair $\{i,j\}$ in one of the following four colors:
\begin{itemize}
    \item[\tiny$\bullet$] \emph{Red} if the pair $(V_i,V_j)$ is not $\eps$-regular;
    \item[\tiny$\bullet$] \emph{White} if the pair $(V_i,V_j)$ is $\eps$-regular and $d(V_i,V_j)<\delta$;
    \item[\tiny$\bullet$] \emph{Grey} if the pair $(V_i,V_j)$ is $\eps$-regular and $\delta\leq d(V_i,V_j)\leq 1-\delta$;
    \item[\tiny$\bullet$] \emph{Black} if the pair $(V_i,V_j)$ is $\eps$-regular and $d(V_i,V_j)> 1-\delta$.
\end{itemize}
Let $R_{\rm grey}:=R_{\rm grey}(P)$ be the graph on
$[t]$ whose edges are the grey pairs. An immediate consequence of Corollary~\ref{cor:triangle-free} and our choice of $\eps$, $\delta$ and $n$ is the following observation.

\begin{claim}\label{claim:grey_triangle_free_main}
The graph \(R_{\rm grey}\) is triangle-free.
\end{claim}

Next, we classify all possible colorings and structure obtained from different posets $P$ through the regularity discussion above. A \emph{template} is a tuple $X=(t,\ell,\omega,\chi)$, where each parameter corresponds to the following:
\begin{itemize}
    \item[\tiny$\bullet$] The integer $t$ corresponds to the size of the $\eps$-partition and $k\leq t \leq M$;
    \item[\tiny$\bullet$] The size of $V_0$ is $|V_0|=\ell$ and $0\leq \ell \leq \eps n$;
    \item[\tiny$\bullet$] The function $\omega:[t]\rightarrow [k]$ is the function such that for each $i\in [t]$, the element $\omega(i)\in [k]$ is the unique index such that $V_i\subseteq W_{\omega(i)}$;
    \item[\tiny$\bullet$] The coloring $\chi:[t]^{(2)}\to \{\mathrm{red},\mathrm{white},\mathrm{grey},\mathrm{black}\}$ is a coloring of the pairs $[t]^{(2)}$ into the four colors described earlier.
\end{itemize}
Clearly, the discussion above gives us a natural map from a poset $P$ to a template $X:=X(P)$. Let $\mathcal{X}$ be the set of all templates. Since $k$ and
$M$ are chosen to be constants, then the number of possible templates is at most
\begin{equation}\label{eq:number_templates}
        |\mathcal{X}|\leq (\eps n+1)\sum_{k\leq t\leq M} k^t 4^{\binom{t}{2}}
        \leq 2^{\eta n/10}
\end{equation}
for sufficiently large $n$.

We are now able to describe the bijection $\psi_P:P\rightarrow [n]$ to be used in our abstract construction. Fixed a $P \in \mathcal{P}_n$, the partitions $\{W_1,
\ldots, W_k\}$, $\{V_0,\ldots,V_t\}$ and template $X:=X(P)=(t,\ell,\omega,\chi)$, we define a bijection $\overline{\omega}:[t]\rightarrow [t]$ with the following two properties:
\begin{itemize}
    \item[\tiny$\bullet$] If $\omega(i)<\omega(j)$, then $\overline{\omega}(i)<\overline{\omega}(j)$.
    \item[\tiny$\bullet$] If $\omega(i)=\omega(j)$ and $i<j$, then $\overline{\omega}(i)<\overline{\omega}(j)$.
\end{itemize}
Note that given $\omega$, the map $\overline{\omega}$ is unique. We now define a partition $[n]=I_0\cup \ldots\cup I_t$ such that $|I_0|=\ell$ and $|I_1|=\ldots=|I_t|=m:=(n-\ell)/t$ and with the property that $I_i<I_j$ for every $i<j$. That is, $I_0,\ldots, I_t$ are consecutive intervals of $[n]$ such that $I_i$ has size $|V_i|$ for $0\leq i \leq t$. The bijection is now defined as an arbitrary map such that $\psi_P(V_0)=I_0$ and $\psi_P(V_i)=I_{\overline{\omega}(i)}$ for $0\leq i \leq t$. That is, if $\overline{\omega}(i)<\overline{\omega}(j)$, then $\psi_P(V_i)<\psi_P(V_j)$. Since the order inside each $V_i$ is chosen arbitrarily, the map $\psi_P$ is not unique, but it has the following properties:
\begin{enumerate}[label=(R\arabic*)]
    \item\label{A:order} If $\overline{\omega}(i)<\overline{\omega}(j)$, then $\psi_P(V_i)<\psi_P(V_j)$.
    \item\label{A:template_uniqueness} If $X(P)=X(P')$, then $\psi_P(V_i)=\psi_{P'}(V_i)$ for $0\leq i \leq t$.
\end{enumerate}
Property~\ref{A:template_uniqueness} is particularly important, since it says that the template is enough to determine the order up to the elements inside the $\eps$-regular partition.

Finally, we finish our preparation by defining an auxiliary graph $G_{\rm aux}:=G_{\rm aux}(P)$ on the vertex set $[n]$ as follows. Let $P$ be a poset with corresponding initial partition $\{W_1, \ldots, W_k\}$, $\eps$-regular partiton $\{V_0,\ldots, V_t\}$ and bijection $\psi_P$. A pair $\{x,y\}\in [n]^{(2)}$ is an edge of $G_{\rm aux}$ if at least one of the following
conditions holds:
\begin{enumerate}[label=(S\arabic*)]
    \item\label{B:block} The elements $\psi_P^{-1}(x)$ and $\psi_P^{-1}(y)$ lie in the part $W_j$, for $j\in [k]$;
    \item\label{B:exceptional} At least one of the vertices $\psi_P^{-1}(x),\psi_P^{-1}(y)$ lies in $V_0$;
    \item\label{B:red} If $\psi_P^{-1}(x)\in V_i$, $\psi_P^{-1}(y)\in V_j$, and the $\{i,j\}$ is of color red;
    \item\label{B:white} If $\psi_P^{-1}(x)\in V_i$, $\psi_P^{-1}(y)\in V_j$, the pair $\{i,j\}$ is of color white, and $\{\psi_P^{-1}(x),\psi_P^{-1}(y)\}\in G_P$;
    \item\label{B:black} If $\psi_P^{-1}(x)\in V_i$, $\psi_P^{-1}(y)\in V_j$, the pair $\{i,j\}$ is of color black, and $\{\psi_P^{-1}(x),\psi_P^{-1}(y)\}\notin G_P$.
\end{enumerate}
In a certain sense, the graph $G_{\rm aux}$ stores all pairs whose behaviour is not forced by the reduced
colored graph $R_{\rm grey}$ and by the order of the blocks $W_j$'s. The next claim shows that $G_{\rm aux}$ is a sparse graph. 

\begin{claim}\label{claim:aux_graph}
$e(G_{\rm aux})\leq 2 n^2/k$.
\end{claim}

\begin{proof}
We analyze the contribution in each condition. By~\eqref{eq:sizeW_i} we have that the number of edges contributed by condition~\ref{B:block} is at most
\begin{align*}
    \sum_{j=1}^k\binom{|W_j|}{2}\leq \frac{1}{2}\sum_{j=1}^k|W_j|^2\leq \frac{n^2}{k}.
\end{align*}
Since $|V_0|\leq \eps n$, the number of edges contributed by condition~\ref{B:exceptional} is at most $\eps n^2$. The condition of $\eps$-regularity says that there are at most $\eps t^2$ red pairs. Hence, if $m=|V_1|=\ldots=|V_t|$, then the contribution of~\ref{B:red} is at most $\eps t^2m^2\leq \eps n^2$. By the definition, each pair $\{i,j\}$ that is white has density at most $\delta$. Therefore, the number of edges contributed by condition~\ref{B:white} is at most $\delta t^2 m^2\leq \delta n^2$. Similarly, each black pair has nonedge density at most $\delta$. Therefore, the number of edges contributed by condition~\ref{B:black} is at most $\delta n^2$. By our choice of $\eps, \delta$ and $k$, we obtain
\begin{align*}
        e(G_{\rm aux})\leq
        \left(\frac{1}{k}+2\eps+2\delta\right)n^2
        \leq \frac{2 n^2}{k}.
\end{align*}
This concludes the proof of the claim.
\end{proof}

We are now able to proof the main theorem.

\begin{proof}[Proof of Theorem \ref{thm:universal-bound}]
    As discussed before, in order to construct our universal poset of Section \ref{sec: poset construction}, we need for each $P \in \mathcal{P}_n$ to define a bijection from $P$ to $[n]$ and an auxiliary set $\varphi(P)\subseteq [n]$. Let $P\in \mathcal{P}_n$ be a fixed poset. Let $\{W_1,\ldots,W_k\}$ be the initial partition, $\{V_0,\ldots, V_t\}$ be the $\eps$-partition obtained in the discussion in the beginning of this section, and let $X:=X(P)$ be the template of $P$. We take the map $\psi_P$ defined earlier in the section to be the bijection from $P$ to $[n]$.

    We now focus on defining the auxiliary set $\varphi(P)\subseteq[n]$. We split the set into two parts. First, define 
\begin{align*}
        S_1:=S_1(P)=\left\{x \in [n]: \: \deg_{G_{\rm aux}}(x)\geq \frac{2n}{\sqrt{k}}\right\}
\end{align*}
to be the set of vertices of $G_{\rm aux}$ of high degree.
By Claim~\ref{claim:aux_graph}, we have
\begin{equation}\label{eq:S1_small}
        |S_1|\leq \frac{4n^2/k}{2n/\sqrt{k}}\leq \frac{2n}{\sqrt{k}}.
\end{equation}

To define the second part of the auxiliary set, we use the triangle-freeness  of the grey reduced graph. Let $i_{\star}\in[t]$ be a vertex of maximum degree in $R_{\rm grey}$. Our goal is to construct a set $S_2$ with $|S_2|\leq n/2$ and the following property for $i\in [t]$:
\begin{align}\label{eq:property_S2}
    \text{if \quad $\psi_P(V_i) \not\subseteq S_2$,\quad\quad\quad  then \quad\quad\quad $\deg_{R_{\rm grey}}(i)\leq t/2$}.
\end{align}
We split the construction in two cases:

\vspace{0.2cm}

\noindent \textbf{Case 1:} $\deg_{R_{\rm grey}}(i_\star)\geq t/2$.

\vspace{0.2cm}

In this case, write $W=N_{R_{\rm grey}}(i_\star)$ as the neighborhood of $i_{\star}$ in the grey reduced graph $R_{\rm grey}$. We define the set $S_2\subseteq [n]$ by  
\begin{align*}
    S_2:=S_2(P)=\bigcup_{i\notin W}\psi_P(V_i) .
\end{align*}
By Claim \ref{claim:grey_triangle_free_main}, the graph $R_{\rm grey}$ is triangle-free. This implies that $W$ is an independent set. Hence, for every $i \in W$, it holds that $\deg_{R_{\rm grey}}(i)\leq \big|[t]
    \setminus W\big|\leq t/2$. This implies
that property~\eqref{eq:property_S2} holds.

\vspace{0.2cm}

\noindent \textbf{Case 2:} $\deg_{R_{\rm grey}}(i_\star)< t/2$.

\vspace{0.2cm}

In this case we just take $S_2=\emptyset$ and property~\eqref{eq:property_S2} holds immediately since $i_\star$ is the vertex of maximum degree.

\vspace{0.2cm}

Finally, we set
\begin{align*}
    \varphi(P):=S=S_1\cup S_2.
\end{align*}

Let $\mathcal{S}=\{\varphi(P):\: P\in \mathcal{P}_n\}$ be the collection of the auxiliary sets. We first estimate the size of $\mathcal{S}$. We bound this size by the number of pairs $(S_1,S_2)$. By~\eqref{eq:S1_small}, we always have $|S_1|\leq 2n/\sqrt{k}$. Hence, by our choice of $k$, there are 
\begin{align}\label{eq:choices_S1}
    \sum_{i=0}^{2n/\sqrt{k}}\binom{n}{i}\leq \left(\frac{en}{2n/\sqrt{k}}\right)^{2n/\sqrt{k}}\leq 2^{\eta n/10}
\end{align}
choices of $S_1$ for sufficiently large $n$. 

To estimate the number of sets $S_2$ we analyze the templates in $\mathcal{X}$. Let $X=(t,\ell,\omega,\chi)\in \mathcal{X}$ be a fixed template and consider all the posets $P\in \mathcal{P}_n$ such that $X(P)=X$. We claim that $S_2$ is uniquely determined by $X$. Note that the grey graph $R_{\rm grey}$ is determined by the reduced coloring $\chi$. If we are in Case 2, then the set $S_2$ is completely determined since $S_2=\emptyset$. Otherwise, the set $S_2=\bigcup_{i\notin W} \psi_P(V_i)$, where $W$ is determined by $\chi$. However, by property~\ref{A:template_uniqueness}, the sets $\psi_P(V_i)$ are uniquely determined by $X$. Therefore, the set $S_2$ is also determined in this case, which completes the proof of the claim. Putting together with~\eqref{eq:number_templates} this gives at most $|\mathcal{X}|\leq 2^{\eta n/10}$
choices for the set $S_2$. Hence, we have that
\begin{align}\label{eq:size_cS}
    |\mathcal{S}|\leq 2^{\eta n/10}\cdot 2^{\eta n/10} \leq 2^{\eta n/5}.
\end{align}

It remains to bound the number of possible profiles. Fix a possible auxiliary set $S=S_1\cup S_2$, a vertex $x\in[n]$, and a template $X=(t,\ell,\omega,\chi)$. We want to  estimate the
number of possible pairs 
\begin{align*}
    \mathcal{T}_{S,x}^{X}=\{(U_x^P,D_x^P):\: P\in \mathcal{P}_n,\, \varphi(P)=S,\, x\in [n],\, X(P)=X\}
\end{align*}
that can arise from posets $P$ with auxiliary set $S$ and template $X$. We start by recalling that $S_1$ has size at most $2n/\sqrt{k}$. Hence, by~\eqref{eq:choices_S1}, for a fixed $S$ there are $2^{\eta n/10}$ choices of $S_1$ and $S_2$. Fix a choice of $S_1$ and $S_2$ for the remaining of the analysis. We split the proof into two cases depending whether $x$ lies in $S$ or not.

\vspace{0.2cm}

\noindent \textbf{Case A:} $x\notin S$.

\vspace{0.2cm}

In this case, it holds that $x\notin S_1$ and $x\notin S_2$. Recall, by definition that $x\notin S_1$ implies that $\deg_{G_{\rm aux}}(x)<2n/\sqrt{k}$. That is, the neighborhood of $x$ in $G_{\rm aux}$ is small. Fix a potential neighborhood $Z:=N_{G_{\rm aux}}(x)\subseteq [n]$ of $x$ and let $P$ run over all possible posets in $\mathcal{P}_n$ with auxiliary sets $S_1$ and $S_2$, template $X(P)=X$ and $Z$ as the neighborhood of $x$ in $G_{\rm aux}$. Note that this requires as well that we estimate the possible pairs $(U_x^P, D_x^P)$ by splitting into two parts: The number of pairs $(U_x^P\cap Z, D_x^P\cap Z)$ and $(U_x^P\cap Z^{\mathsf{c}}, D_x^P\cap Z^{\mathsf{c}})$. We start with the pairs $(U_x^P\cap Z, D_x^P\cap Z)$. Note that if $y\in Z$, then there are at most three possibilities of containment: either $y\in U_x^P$, $y\in D_x^P$ or $y\notin U_x^P\cup D_x^P$. Therefore, there are at most $3^{|Z|}$ different possibilities for the pair $(U_x^P\cap Z, D_x^P\cap Z)$.

We now estimate the number of pairs $(U_x^P\cap Z^{\mathsf{c}}, D_x^P\cap Z^{\mathsf{c}})$. Let $y\notin Z$. If $\psi^{-1}_P(x)\in V_0$, then by condition~\ref{B:exceptional} we have that $\deg_{G_{\rm aux}}(x)=n-1>2n/\sqrt{k}$. This contradicts the fact that $x\notin S_1$. Thus, we may assume that $\psi^{-1}_P(x)\in V_i$ for $i\neq 0$. Suppose that $\psi_P^{-1}(y)\in V_j$. Since $y\notin Z$, by conditions~\ref{B:block} and~\ref{B:red}, we have that $j\neq i$ and $\{i,j\}$ is not of color red. There are three possibilities for the color of $\{i,j\}$ and we analyze whether $y \in U_x^P$, $y\in D_x^P$ or $y\notin U_x^P\cup D_x^P$ for each case:

\vspace{0.2cm}

\noindent \textbf{Case A.1:} $\{i,j\}$ is of color white.

\vspace{0.2cm}

By condition~\ref{B:white} we have that $\{\psi^{-1}_P(x),\psi_P^{-1}(y)\} \notin G_P$. Hence, $\psi^{-1}_P(x)$ and $\psi^{-1}_P(y)$ are not comparable in $P$ and we have $y\notin U_x^P\cup D_x^P$.

\vspace{0.2cm}

\noindent \textbf{Case A.2:} $\{i,j\}$ is of color black.

\vspace{0.2cm}

By condition~\ref{B:black} we have that $\{\psi^{-1}_P(x),\psi^{-1}_P(y)\} \in G_P$. Moreover, by condition \ref{B:block}, we have that $\psi^{-1}_P(x)$ and $\psi^{-1}_P(y)$ lies in different classes of the partition $W_1\cup\ldots \cup W_k$. Hence, by~\eqref{eq:star_property}, whether $y\in U_x^P$ or $y\in D_x^P$ is completely determined.

\vspace{0.2cm}

\noindent \textbf{Case A.3:} $\{i,j\}$ is of color grey. 

\vspace{0.2cm}

In this case, it is possible that either $\{\psi^{-1}_P(x),\psi^{-1}_P(y)\}$ is an edge or a non-edge in $G_P$. In the latter case, we have that $y\notin U_x^P\cup D_x^P$. In the case that it is an edge, then by the same observation as in Case A.2, whether $y\in U_x^P$ or $y\in D_x^P$ is determined. Therefore, for a fixed $y$ in this case, there are only two possibilities of containment. 

Since $x\notin S_2$, by~\eqref{eq:property_S2}, the number of $j$ such that $\{i,j\}$ is grey is at most $t/2$. Moreover, the neighborhood $N_{R_{\rm grey}}(i)$ is determined by $\chi$ and by property~\ref{A:template_uniqueness} the set $C:=\bigcup_{j\in N_{R_{\rm grey}}(i)} \psi_P(V_j)$ is also determined by the template $X$. Hence, the total contribution over all possible $y$ in this case is at most $2^{|C|}\leq 2^{m\cdot t/2}\leq 2^{n/2}$.

\vspace{0.2cm}

Putting everything together, we have that the number of pairs $(U_x^P,D_x^P)$ if $Z$ and the partition $S=S_1\cup S_2$ are fixed is at most $3^{|Z|}\cdot 2^{n/2}$. By summing over all possible neighborhoods $Z$ and choices of $S_1$, we obtain
\begin{align}\label{eq:caseA}
    |\mathcal{T}_{S,x}^X|\leq 2^{\eta n/10}\cdot 2^{n/2}\cdot \sum_{z\leq 2n/\sqrt{k}} \binom{n}{z}3^{z}\leq 2^{n/2+\eta n/5}
\end{align}
for sufficiently large $n$.

\vspace{0.2cm}

\noindent \textbf{Case B:} $x\in S$.

\vspace{0.2cm}

Again, in this case let $P$ run over all the posets with auxiliary sets exactly $S_1$ and $S_2$ and template $X=(t,\ell, \omega, \chi)$. Since $x\in S$, the local profile only stores the pair $(U_x^P\cap S, D_x^P\cap S)$. By construction, the set $S_1$ has size at most $2n/\sqrt{k}$, while the set $S_2$ has size at most $n/2$. This implies that
\begin{align}\label{eq:size_S}
    |S|\leq |S_1|+|S_2|\leq n/2+ 2n/\sqrt{k}\leq n/2+\eta n/10
\end{align} 
for sufficiently large $n$. Suppose that $\psi_{P}^{-1}(x) \in W_a$ for some $a \in [k]$. Note that by property~\ref{A:template_uniqueness} the sets $\psi_P(V_i)$ are determined by the template $X$ for every $0\leq i \leq t$. In particular, the sets $\psi_P(V_i)$ are determined for $i\in \omega^{-1}(a)$, i.e., for every $V_i\subseteq W_a$ with $i\neq 0$. Hence, the set $\psi_{P}(W_a\setminus V_0)$ is determined by the template $X$. 

If $x\in \psi_P(V_i)$ for some $i\neq 0$, then the previous discussion implies that $a$ is determined by the template. Otherwise, if $x\in \psi_P(V_0)$, then the choice of $a$ depends on the poset $P$. There are $k$ possibilities for $a$. Fix one of them. Let $Z=\psi_P(W_a\cap V_0)\subseteq \psi_P(V_0)$. Since $\psi_P(V_0)$ is determined and $|V_0|=\ell$, there are $2^\ell$ possibilities for the set $Z$. Fix one of them. We split the vertices $y\in S$ into two possibilities. For each possibility, we will analyze whether $y\in U_x^P$, $y\in D_x^P$ or $y\notin U_x^P\cup D_x^P$.

\vspace{0.2cm}

\noindent \textbf{Case B.1:} $y\in S \setminus \psi_P(W_a)$.

\vspace{0.2cm}

This case is somewhat similar to case A.3. Note that since $y\notin \psi_P(W_a)$, then $\psi_P^{-1}(x)$ and $\psi_P^{-1}(y)$ lies in different classes of the partition $W_1\cup\ldots\cup W_k$. There are two cases: either $\{\psi^{-1}_P(x),\psi^{-1}_P(y)\}$ is a non-edge or an edge of $G_P$. In the former case, we have that $y\notin U_x^P\cup D_x^P$. Otherwise, in the case that it is an edge of $G_P$, we have by property~\eqref{eq:star_property} that whether $y\in U_x^P$ or $y\in D_x^P$ is completely determined. This implies that there are only two possibilities of containment. Hence, by~\eqref{eq:size_S} the total contribution in this case is at most $2^{|S\setminus \psi_P(W_a)|}\leq 2^{|S|}\leq 2^{n/2+\eta n/10}$.

\vspace{0.2cm}

\noindent \textbf{Case B.2:} $y\in S \cap \psi_P(W_a)$.

\vspace{0.2cm}

This case we do not make any restriction and all three possible containments are allowed. Hence, by~\eqref{eq:sizeW_i}, the contribution is at most $3^{|S\cap \psi_P(W_a)|}\leq 3^{|W_a|}\leq 3^{2n/k}\leq 2^{\eta n/10}$ for sufficiently large $n$.

\vspace{0.2cm}

Putting the two cases together, we have that the number of pairs $(U_x^P,D_x^P)$ for $a$ and $Z$ fixed and partition $S=S_1\cup S_2$ is at most $2^{n/2+\eta n/10}\cdot 2^{\eta n/10}\leq 2^{n/2+\eta n/5}$. By summing over all possible values of $a$, $Z$ and $S_1$, we obtain
\begin{align}\label{eq:caseB}
    |\mathcal{T}_{S,x}^X|\leq k\cdot 2^{\ell}\cdot 2^{\eta n/10}\cdot 2^{n/2+\eta n/5}\leq  2^{n/2+3\eta n/10+\eps n}\leq 2^{n/2+\eta n/3}.
\end{align}

By \eqref{eq:number_templates}, \eqref{eq:caseA} and \eqref{eq:caseB}, we conclude that for every possible auxiliary set $S$ and every $x\in[n]$,
\begin{equation}\label{eq:profile_bound_final}
        |\mathcal T_{S,x}|\leq \sum_{X\in \mathcal{X}} |\mathcal{T}_{S,x}^X|
        \leq |\mathcal{X}|\cdot 2^{n/2+\eta n/3}\leq 2^{n/2+\eta n/2}.
\end{equation}

We now apply Proposition~\ref{prop:construction_size}. By \eqref{eq:size_cS} and
\eqref{eq:profile_bound_final}, the universal poset $Q:=Q(\mathcal{P}_n,\varphi)$ constructed in
Section~\ref{sec: poset construction} satisfies
\begin{align*}
        |Q|\leq n\cdot |\mathcal{S}|\cdot \max_{S\in \mathcal{S}, x\in [n]}|\mathcal{T}_{S,x}|\leq n\cdot 2^{\eta n/5}\cdot 2^{n/2+\eta n/2}\leq 2^{n/2+\eta n}
\end{align*}
for sufficiently large $n$. This finishes the proof of Theorem~\ref{thm:universal-bound}.
\end{proof}

\section{Concluding remarks}

It is natural to conjecture that the bound in Theorem~\ref{thm:universal-bound} is not best possible. In view of the counting lower bound coming from the theorem of Kleitman and Rothschild~\cite{Kleitman1975Posets}, the natural target is $f(n)=2^{(1+o(1))n/4}$. There are some reasons to believe that this should be the correct order of magnitude. Alon~\cite{Noga2017asymptotically} proved that there exists an induced-universal graph on $(1+o(1))2^{n/4}$ vertices for the family of $n$-vertex bipartite graphs, and also showed that this bound is asymptotically optimal. Since every bipartite graph with bipartition $A\cup B$ can be viewed as a two-layered poset by declaring $a<_P b$ precisely when $a\in A$, $b\in B$, and $\{a,b\}$ is an edge, this result covers the problem at least in the two-layered case.

The next natural test case is the family of three-layered posets. Indeed, Kleitman and Rothschild~\cite{Kleitman1975Posets} proved that almost all posets are three-layered. Thus, an interesting intermediate problem is to determine whether the family of $n$-element three-layered posets admits a universal poset of size $2^{(1+o(1))n/4}$. A positive answer would already capture the typical structure of finite posets and would be a step toward the conjectured bound for $f(n)$.

%One possible strategy for obtaining a smaller universal poset than that of Theorem~\ref{thm:universal-bound} would be to develop an adjacency labeling scheme based on an edge-coloring technique similar to that of~\cite{Lozin2007Minimal}, combined with an application of the regularity lemma. The principal challenge, however, is to ensure that the resulting structure into which the family is embedded retains the properties of a partially ordered set.

%{\bf Update reference:}

%P. Bastide, C. Groenland, and R. Nenadov, Smaller universal posets [arXiv]
%Combinatorial Theory, accepted.
%\RG{The paper has been accepted but the reference is not available yet.}
%{\bf JB: "Combinatorial Theory, to appear. That is proper citation.}

%{\bf JB: it is a CS paper, still should be fully cited.}

%[8]: STOC 2021: Proceedings of the 53rd Annual ACM SIGACT Symposium on Theory of Computing
%Pages 1109 - 1117\\

\printbibliography

\end{document}